\documentclass[a4paper]{amsart}

\usepackage{amssymb}
\usepackage[margin=3.5cm]{geometry}
\usepackage{color}
\usepackage{hyperref}
\usepackage{enumitem}

\usepackage{needspace} 

\usepackage{mathtools}
\mathtoolsset{centercolon}

\DeclareMathOperator{\Aut}{Aut}

\DeclareMathOperator{\GL}{GL}

\DeclareMathOperator{\PSL}{PSL}

\DeclareMathOperator{\Mat}{Mat}

\newtheorem{theorem}{Theorem}[section]
\newtheorem{lemma}[theorem]{Lemma}
\newtheorem{corollary}[theorem]{Corollary}
\newtheorem{proposition}[theorem]{Proposition}

\theoremstyle{definition}
\newtheorem{remark}[theorem]{Remark}
\newtheorem{remarks}[theorem]{Remarks}

\newtheorem{definition}[theorem]{Definition}
\newtheorem{example}[theorem]{Example}

\newtheorem*{question*}{Question}
\newtheorem*{remark*}{Remark}

\title{Complete reducibility: Variations on a theme of Serre}
\author{Maike Gruchot, Alastair Litterick, Gerhard R\"{o}hrle}

\address{Maike Gruchot: Lehrstuhl f\"{u}r Algebra und Zahlentheorie, RWTH Aachen University, Pontdriesch 14/16, D-52062 Aachen, Germany}
\address{Alastair Litterick: Department of Mathematical Sciences, University of Essex, Wivenhoe Park, Colchester, Essex CO4 3SQ, United Kingdom}
\address{Gerhard R\"{o}hrle: Fakult\"{a}t f\"{u}r Mathematik, Ruhr-Universit\"{a}t Bochum, Universit\"{a}tsstra{\ss}e 150, D-44780 Bochum, Germany}
\email{maike.gruchot@rwth-aachen.de}
\email{a.litterick@essex.ac.uk}
\email{gerhard.roehrle@rub.de}

\subjclass[2010]{51E24, 20E42, 20G15, 20G40, 14L24}
\keywords{Complete reducibility in spherical buildings, $G$-complete reducibility}

\begin{document}

\begin{abstract}
In this note, we unify and extend various concepts in the area of $G$-complete reducibility, where $G$ is a reductive algebraic group. By results of Serre and Bate--Martin--R\"{o}hrle, the usual notion of $G$-complete reducibility can be re-framed as a property of an action of a group on the spherical building of the identity component of $G$. We show that other variations of this notion, such as relative complete reducibility and $\sigma$-complete reducibility, can also be viewed as special cases of this building-theoretic definition, and hence a number of results from these areas are special cases of more general properties.
\end{abstract}

\maketitle

\section{Introduction}

This paper concerns the notion of complete reducibility in the theory of reductive algebraic groups. Let $G$ be a reductive algebraic group, defined over an algebraically closed field, and to begin let us assume that $G$ is connected. Following Serre \cite{Serre1998}, a subgroup $H$ of $G$ is called \emph{$G$-completely reducible} if, whenever $H$ is contained in a parabolic subgroup $P$ of $G$, it is contained in a Levi subgroup of $P$. This fundamental notion directly generalises the concept of a completely reducible $H$-module (the case that $G$ is a general linear group), and has proved highly fruitful in streamlining results in the theory of algebraic groups. As well as having direct applications to the subgroup structure of $G$, see for instance \cite{Litterick2018,Stewart2013}, it relates closely to the semisimplicity of subgroups on $G$-modules \cite[\S 5.2]{Serre2003-2004}, and provides connections with the related notion of \emph{strongly reductive} subgroups and geometric invariant theory \cite{Richardson1982,Bate2005}.

Because of these connections, complete reducibility extends naturally to non-connected reductive groups $G$ \cite[\S 6]{Bate2005}, replacing parabolic subgroups and Levi subgroups with so-called R-parabolic subgroups and R-Levi subgroups (see Section~\ref{sec:prelims} for definitions). By placing restrictions on the parabolics and Levis under consideration, we also obtain generalisations which we shall discuss shortly. The purpose of this note is to show that these generalisations are all special cases of Serre's original building-theoretic concept, and that results in each area can therefore be viewed as special cases of more general phenomena.

Let $X$ be an arbitrary spherical building, \cite{Tits74}. A subset $Y$ of
$X$ is said to be \emph{convex} if whenever two points of  
$Y$ are not opposite in $X$, then  $Y$  
contains the unique geodesic joining these points.
The vertices of $X$ can be labelled in an 
essentially unique way via an equivalence
relation, cf.~\cite[\S 2.1.2]{Serre2003-2004}; 
the \emph{type} of a vertex is its label.
An automorphism $\alpha$ of $X$ is said to be \emph{type-preserving}
if $x$ and $\alpha(x)$ have the same type for all vertices $x$ of $X$.

Now let $X = X(G^{\circ})$ be the spherical Tits building of the identity component $G^{\circ}$ of 
a reductive algebraic group $G$, \cite{Tits74}.
Recall that the simplices in $X$ correspond to
the parabolic subgroups of $G^{\circ}$ and the vertices
of $X$ correspond to the maximal proper parabolic subgroups  of $G^{\circ}$, see \cite[\S 3.1]{Serre2003-2004}.

By an action of a group $\Gamma$ on a spherical building $X$ we mean an action on $X$ by simplicial building automorphisms.
In that case let $X^\Gamma$ be the fixed point subset of the action of $\Gamma$, i.e.\ 
the subset of all $\Gamma$-stable (thus $\Gamma$-fixed) simplices in $X$ and note that 
this subset is always convex. 
If the action of  $\Gamma$ is type-preserving then 
$X^\Gamma$ is a subcomplex, else it is only a subcomplex of the barycentric subdivision of $X$, 
\cite[\S 2.3.1]{Serre2003-2004}.

By an action of a group $\Gamma$ on the reductive group $G$, we mean an action such that the induced action on 
the spherical building of $G^{\circ}$ is simplicial.
For $G$ connected, abstract automorphisms of $G$ do give rise to such automorphisms on the building of $G$,
see \cite{Steinberg74}.

Our starting point is the following fundamental definition of Serre.

\begin{definition}[{\cite[\S 2]{Serre2003-2004}}] \label{def:cr-action} 
	Let $X$ be a spherical building. 
	A convex subset $Y$ is 
	\emph{$X$-completely reducible} (\emph{$X$-cr} for short) if
	for every $y\in Y$ there
	exists a point $y'\in Y$ opposite to $y$ in $X$.
	An action of a group $\Gamma$ on $X$ is called \emph{completely reducible} if the fixed-point subset $X^{\Gamma}$ is completely reducible. 
	If $X = X(G^{\circ})$ is the spherical building of $G^{\circ}$ for a reductive group $G$, then an action of $\Gamma$ on $G$ is called \emph{completely reducible} if the induced action on $X$ is completely reducible.
\end{definition}

In \cite{Serre2003-2004}, Serre concentrated mainly on the case that $Y = X(G^{\circ})^{\Gamma}$ with $\Gamma$ acting by type-preserving automorphisms. It is useful to relax this condition, for instance if $G$ is not connected then $G$ itself can induce automorphisms of $X(G^{\circ})$ which are not type-preserving; the next example is a basic illustration of this. However even when $G$ is not connected, it turns out that a subgroup $H$ of $G$ is $G$-completely reducible in the usual sense if and only if $X(G^{\circ})^{H}$ is completely reducible \cite[Proposition 6.16]{Bate2005}.

\begin{example} \label{ex:gln}
Let $G := G^\circ \!\!\left<\gamma\right> \le \Aut(G^\circ)$, where $G^\circ$ is simple of adjoint type and  $\gamma$ is a non-trivial graph automorphism of $G^\circ$. Then 
the induced action of $\Gamma = \left<\gamma\right>$ on the building  $X(G^{\circ})$ of $G^\circ$ is not type-preserving.
Thanks to \cite[Theorem 7.2]{Steinberg68}, $\gamma$ stabilizes a Borel subgroup of $G^\circ$. If $\gamma$ also stabilizes a maximal torus of a stable Borel subgroup it is called  \emph{quasi-semisimple}. 
In that case the $G$-conjugacy class of $\gamma$ is closed in $G$, owing to \cite[Corollaire II 2.22]{Spaltenstein82}. It follows from \cite[Theorem 9.3]{Bate2017} that $\Gamma = \left<\gamma\right>$ is $G$-completely reducible.
Thus $\Gamma$ acts completely reducibly on $X(G^{\circ})$
in the sense of Definition~\ref{def:cr-action},
by \cite[Proposition 6.16]{Bate2005}.

In general there may
be elements in a non-connected group $G$ that induce non-quasi-semisimple 
automorphisms of $G^\circ$.
For instance, 
let $G = \Aut(G^\circ)$, where $G^\circ$ is an adjoint simple group of type $D_4$ over an algebraically closed field of characteristic $3$. 
There are exactly two conjugacy
classes of cyclic groups of order three generated by outer automorphisms in $G$. Let 
$\gamma_1$ and $\gamma_2$
be representatives of the respective unipotent $G$-classes, as in 
\cite[Proposition 4.1]{Bate2015}.
They both act non-type preservingly on $X(G^{\circ})$.
By \cite[Proposition 4.1]{Bate2015}, the $G$-class of $\gamma_1$ is closed while the 
$G$-class of $\gamma_2$ is not (it contains the former in its closure). So again by
\cite[Corollaire II 2.22]{Spaltenstein82},
$\gamma_1$ is quasi-semisimple
while $\gamma_2$ is not.
It follows from 
\cite[Theorem 9.3]{Bate2017} that $\Gamma_1 = \left<\gamma_1\right>$ is $G$-completely reducible
while $\Gamma_2 = \left<\gamma_2\right>$ is not. 
Thus $\Gamma_1$ acts completely reducibly on $X(G^{\circ})$
and $\Gamma_2$ does not.
\end{example}

We now come to some variations on the idea of complete reducibility. In the first instance, suppose that $G$ is not necessarily connected, and let $X(G)$ denote the poset of R-parabolic subgroups of $G$ (see \S \ref{sec:prelims}), ordered by reverse inclusion. If $G$ is connected then R-parabolic subgroups coincide with parabolic subgroups, see \cite[\S 6]{Bate2005} for a detailed discussion. While akin to a building, in this case $X(G)$ fails to be a spherical building, or even a simplicial complex  in general (cf.\ \cite[Example 2.3]{Bate2019}). However, the crucial notion of \emph{opposition} in $X(G)$ still makes sense: Two R-parabolic subgroups $P,Q \in X(G)$ are opposite if and only if $P \cap Q$ is an R-Levi subgroup of each. Thus the concept of complete reducibility extends naturally to $X(G)$.

Further variations of $G$-complete reducibility arise by restricting attention to certain convex subsets of $X(G^{\circ})$ and their fixed points under $\Gamma$. 
In one direction, since $G$-complete reducibility for non-connected groups is naturally defined in terms of cocharacters of $G$, \cite[\S 6]{Bate2005}, by restricting attention to cocharacters of an arbitrary reductive subgroup $K$ of $G$, one arrives at the notion of \emph{relative $G$-complete reducibility with respect to $K$}, see Definition \ref{def:relativeCR}.

In another vein, 
suppose $G$ is connected and defined over a finite field and thus equipped with a 
\emph{Steinberg endomorphism}
$\sigma$, i.e.\ 
a surjective endomorphism of $G$ 
that fixes only finitely many points, see \cite{Steinberg68} for a detailed discussion.
Restricting attention to $\sigma$-stable parabolic subgroups of $G$ and $\sigma$-stable Levi subgroups thereof, 
we come to the notion of \emph{$\sigma$-complete reducibility} from \cite{Herpel2011}, 
see Definition \ref{def:sigma-cr}.
Note that $\sigma$ induces an automorphism of $X(G^{\circ})$
also denoted by $\sigma$ (which need not be type-preserving). 

Our main result is now as follows.

\begin{theorem} \label{thm:generalisation} %
Let $G$ be a (possibly non-connected) reductive algebraic group over an algebraically closed field.
\begin{enumerate}[label=(\roman*)]
\item The action of an abstract group $\Gamma$ on $G$ is completely reducible in the sense of Definition~\ref{def:cr-action} if and only if every member of $X(G)^{\Gamma}$ has an opposite in $X(G)^\Gamma$. \label{thm:nonconnected}
\item If $H \le N_G(K^{\circ})$ for a reductive subgroup $K$ of $G$, then $H$ is relatively $G$-completely reducible with respect to $K$ if and only if the induced action of $H$ on $X(K^{\circ})$ is completely reducible. \label{thm:gen-rel}
\item Let $G$ be connected and let $\sigma$
be a Steinberg endomorphism of $G$. Let $H$ be a subgroup of $G$ and let $\Gamma$ be the subgroup of $\Aut(X(G))$ generated by $\sigma$ and the image of $H$. Then $H$ is $\sigma$-completely reducible if and only if $\Gamma$ acts completely reducibly on $X(G)$. \label{thm:gen-sigma}
\end{enumerate}
\end{theorem}

\needspace{2\baselineskip}
\begin{remarks}  \leavevmode 
\begin{itemize}
\item In the case that $\Gamma$ is the image of a subgroup $H$ of $G$ acting by conjugation, part~\ref{thm:nonconnected} recovers \cite[Proposition~6.16]{Bate2005}. Our proof is similar, although we must stick to the language of \emph{normalising} an R-parabolic subgroup or R-Levi subgroup, rather than being contained in it.
\item In part \ref{thm:nonconnected}, one might reasonably expect that some information is lost when passing from an action on the poset $X(G)$ to the induced action on the building $X(G^{\circ})$; our result tells us this is not so, as the natural definition of `completely reducible action on $X(G)$' is equivalent to Definition~\ref{def:cr-action}.
\item In part \ref{thm:gen-sigma}, we emphasize that there are no conditions whatsoever on the subgroup $H$. By \cite[Theorem 1.4]{Herpel2011} (see Theorem \ref{thm:sigma-cr} below), in the special case that $H$ is $\sigma$-stable, it is $G$-completely reducible if and only if it is $\sigma$-completely reducible: That is, $H$ acts completely reducibly on $X(G)$ if and only if $\Gamma$ acts completely reducibly on $X(G)$. Thus Theorem~\ref{thm:generalisation}\ref{thm:gen-sigma} generalises \cite[Theorem 1.4]{Herpel2011} to subgroups $H$ which are not $\sigma$-stable.
\end{itemize}
\end{remarks}

\section{Preliminaries} \label{sec:prelims}

\subsection{Notation and background} Throughout, we let $k$ be an algebraically closed field, and $G$ is a possibly non-connected reductive algebraic group over $k$. We let $Y(G)$ denote the cocharacter group of $G$, consisting of morphisms of algebraic groups $k^{\ast} \to G$. The \emph{limit} of a morphism $\phi \colon k^{\ast} \to G$ is a morphism $\widehat{\phi} \colon k \to G$ extending $\phi$, when this exists, and in this case we write $\lim_{a \to 0} \phi(a) := \widehat{\phi}(0)$. Each $\lambda \in Y(G)$ determines an \emph{R-parabolic subgroup} of $G$ via $P_\lambda:=\{g\in G\mid \lim_{a\to 0} \lambda(a) \cdot g \text{ exists}\}$, where the dot denotes left-conjugation of $G$ on itself, and the \emph{R-Levi subgroup} of $G$ corresponding to $\lambda$ is $L_\lambda:=\{g\in G\mid \lim_{a\to 0} \lambda(a) \cdot g = g\}$. We have $L_{\lambda} = C_{G}(\lambda(k^{\ast}))$. If $P$ is an R-parabolic subgroup of $G$, then by an \emph{R-Levi subgroup of $P$}, we mean a subgroup $L_{\lambda}$ such that $P = P_{\lambda}$. We still have a semidirect product decomposition $P_{\lambda} = R_{u}(P_{\lambda}) L_{\lambda}$, with $R_u(P_\lambda) = \{ g \in G \mid \lim_{a \to 0} \lambda(a) \cdot g = 1 \}$.

Let $X(G)$ denote the poset of R-parabolic subgroups of $G$ under reverse inclusion. R-parabolic subgroups of $G$ are parabolic in the sense that $G/P$ is a complete variety, but the converse is false if $G$ is not connected \cite[\S 6]{Bate2005}. As mentioned in the introduction, the notion of opposition still makes sense for $X(G)$: Two R-parabolic subgroups $P$ and $Q$ are opposite if they intersect in a common R-Levi subgroup. This permits a natural definition of complete reducibility of subsets of $X(G)$.

\subsection{Completely reducible actions and equivariant morphisms.} 

First we prove the counterpart of \cite[Lemma 2.12]{Bate2005} in our more general setting of completely reducible actions. This is a vital ingredient for proving part~\ref{thm:gen-rel} of Theorem~\ref{thm:generalisation}.
Recall that a homomorphism of reductive algebraic groups is said to be \emph{non-degenerate} provided
the identity component of its kernel is a torus.

\begin{lemma} \label{lem:equivariant}
	Let $G_1$ and $G_2$ be reductive algebraic groups, and let $\Gamma$ act on $G_1$ and $G_2$. Let $f \colon G_1 \to G_2$ be a surjective $\Gamma$-equivariant homomorphism.
	\begin{enumerate}[label=(\roman*)]
		\item If $\Gamma$ acts completely reducibly on $G_1$ then $\Gamma$ acts completely reducibly on $G_2$. \label{equivariant-i}
		\item If $f$ is non-degenerate then $\Gamma$ acts completely reducibly on $G_1$ if and only if $\Gamma$ acts completely reducibly on $G_2$. \label{equivariant-ii}
	\end{enumerate}
\end{lemma}

\begin{proof} We mirror the proof of \cite[Lemma 2.12]{Bate2005} (cf.\ also \cite[\S 6.2]{Bate2005}). Let $N = \ker f$. By \cite[Lemma 6.14]{Bate2005}, there exists a subgroup $M$ of $G$ such that $G = MN$, $M \cap N$ is finite and $M$ and $N$ commute. For part \ref{equivariant-i}, suppose that $\Gamma$ acts completely reducibly on $G_1$. Let $\mu \in Y(G_2)$ such that $\Gamma$ stabilises $P_{\mu}(G_2)$. We can write $n\mu = f \circ \lambda$ for some $\lambda \in Y(M)$ and some sufficiently large integer $n$. By \cite[Lemma 6.15(ii)]{Bate2005}, 
	\[P_{\lambda}(G_1) = f^{-1}(f(P_\lambda(G_1))) = f^{-1}(f(\Gamma \cdot P_{\lambda}(G_1))) \supseteq \Gamma \cdot P_\lambda(G_1).\] 
	Since $\Gamma$ acts completely reducibly on $G_1$, there exists some $u \in R_u(P_\lambda(G_1))$ such that 
	$\Gamma \cdot L_{u \cdot \lambda} \subseteq L_{u \cdot \lambda}(G_1)$. By \cite[Lemma 6.15(i)]{Bate2005}, \[\Gamma \cdot L_{f(u) \cdot \mu}(G_2) = f(\Gamma \cdot L_{u \cdot \lambda}) = f(L_{u \cdot \lambda}) = L_{f(u) \cdot \mu}(G_2).\] Thus $\Gamma$ acts completely reducibly on $G_2$.
	
	\ref{equivariant-ii} Suppose now that $\Gamma$ acts completely reducibly on $G_2$. Let $\lambda \in Y(G_1)$ such that $\Gamma \cdot P_\lambda \subseteq P_\lambda$. Then $P_{f \circ \lambda} = f(\Gamma \cdot P_\lambda) = \Gamma \cdot P_{f \circ \lambda}$, by \cite[Lemma 6.14(i)]{Bate2005}. Since $\Gamma$ acts completely reducibly on $G_2$, there exists $u \in R_u(P_{f \circ \lambda})$ such that $\Gamma \cdot L_{u \cdot (f \circ \lambda)} \subseteq L_{u \cdot (f \circ \lambda)}$. By \cite[Lemma 6.15(iv)]{Bate2005}, there exists $u_1 \in R_u(P_\lambda)$ such that $f(u_1) = u$. We then have 
	\[\Gamma \cdot L_{u_1 \cdot \lambda} \subseteq f^{-1}(f(\Gamma \cdot L_{u_1 \cdot \lambda})) = f^{-1}(L_{u \cdot (f \circ \lambda)}) = NL_{u_1 \cdot \lambda} = L_{u_1 \cdot \lambda},\] 
	since $N$ is contained in every R-Levi subgroup of $G_1$. Therefore, $\Gamma$ acts completely reducibly on $G_1$.
\end{proof}

\section{Proof of Theorem~\ref{thm:generalisation}}

We prove the individual parts of Theorem~\ref{thm:generalisation} separately, since the exact nature of the acting group changes in each scenario. 

\subsection{\texorpdfstring{$G$}{G}-complete reducibility and the building of \texorpdfstring{$G^{\circ}$}{G°}.} Before proving part~\ref{thm:nonconnected} of Theorem~\ref{thm:generalisation}, we require one more piece of set-up. It is easily seen that 
for an R-parabolic subgroup $P$ of $G$ we have
$P^{\circ} = G^{\circ} \cap P$ and $R_{u}(P) = R_{u}(P^{\circ})$. Moreover, thanks to \cite[Proposition 5.4(a)]{Martin03}, the normaliser $N_G(P^{\circ})$ is also an R-parabolic subgroup of $G$.

\begin{proposition} \label{prop:G-cr-G0-cr} %
The action of an abstract group $\Gamma$ on $G$ is completely reducible if and only if every member of $X(G)^{\Gamma}$ has an opposite in $X(G)^\Gamma$.
\end{proposition}

\begin{proof}
Suppose that every member of $X(G)^{\Gamma}$ admits an opposite in $X(G)^{\Gamma}$ and let $P \in X(G^{\circ})^{\Gamma}$. Then $N_G(P)$ is an R-parabolic subgroup of $G$ \cite[Proposition 5.4(a)]{Martin03} which is stabilised by $\Gamma$, hence by assumption it admits an opposite in $X(G)^{\Gamma}$. But then the identity component of this parabolic subgroup is a $\Gamma$-stable opposite to $P$ in $X(G^{\circ})$, as required.

Conversely, suppose that $\Gamma$ acts completely reducibly on $X(G^{\circ})$, and let $P$ be an R-parabolic subgroup  in $X(G)^{\Gamma}$. Then $P^{\circ}$ is also $\Gamma$-stable, so by assumption there exists an opposite parabolic subgroup $Q \in X(G^{\circ})^{\Gamma}$. Then the proof of \cite[Proposition~6.16]{Bate2005} shows that $N_G(Q)$ in $X(G)^{\Gamma}$ meets $P$ in a ($\Gamma$-stable) R-Levi subgroup, hence is the required opposite to $P$  in $X(G)^{\Gamma}$.
\end{proof}

Although Proposition~\ref{prop:G-cr-G0-cr} is perhaps a natural statement to expect, it implies at once the following statement for subgroups of $G$, which is not so clear \emph{a priori} (although it can be derived from known results: the reverse implication is given in \cite[Corollary 2.5]{Bate2015}, while the forward implication follows from a special case of \cite[Lemma~5.1]{Bate2010}).
\begin{corollary}
Let $H$ be a subgroup of a reductive algebraic group $G$. Then $H$ is $G$-completely reducible if and only if, whenever $H$ normalises an R-parabolic subgroup $P$ of $G$, it normalises an R-Levi subgroup of $P$.
\end{corollary}

\subsection{Relative complete reducibility} \label{sec:rel} 
First we recall \cite[Definition 3.1]{Bate2011}.

\begin{definition}
	\label{def:relativeCR}
	Let $H$ and $K$ be subgroups of $G$ with $K$ reductive. We say that $H$ is \emph{relatively $G$-completely reducible with respect to $K$} if, whenever $\lambda \in Y(K)$ such that $H \le P_{\lambda}$, there exists $\mu \in Y(K)$ with $P_{\lambda} = P_{\mu}$ and $H \le L_{\mu}$.
\end{definition}

Armed with Lemma~\ref{lem:equivariant} we can now address
part~\ref{thm:gen-rel} of Theorem~\ref{thm:generalisation}.

\begin{proposition} \label{prop:rel-G-cr}
With the above notation, suppose $H$ is a subgroup of $G$ which normalises $K^{\circ}$. Then $H$ is relatively $G$-completely reducible with respect to $K$ if and only if the conjugation action of $H$ on $K^{\circ}$ is completely reducible.
\end{proposition}

\begin{proof}
It is straightforward to see that $H$ is relatively $G$-cr with respect to $K$ if and only if $H$ is relatively $G$-cr with respect to $K^{\circ}$. By this and Proposition~\ref{prop:G-cr-G0-cr}, it suffices to assume that $K = K^{\circ}$.

Let $N = N_G(K)$, $C = C_G(K)$, and $\pi \colon N \to N/C$ be the natural quotient map. From \cite[Theorem 1]{Gruchot2019}, $H$ is relatively $G$-cr with respect to $K$ if and only if $\pi(H)$ is $\pi(N)$-cr. Now by Theorem~\ref{thm:generalisation}\ref{thm:nonconnected}, $\pi(H)$ is $\pi(N)$-cr if and only if the action of $\pi(H)$ on $\pi(N_G(K))$ is completely reducible. Now, the connected kernel of the induced map $K \to \pi(K)$ is $Z(K)^{\circ}$, a torus, and so by Lemma~\ref{lem:equivariant}\ref{equivariant-ii}, $H$ acts completely reducibly on $K$ if and only if $H$ (equivalently, $\pi(H)$) acts completely reducibly on $\pi(K)$. The result follows.
\end{proof}

\begin{remark} 
	Notice that in the setting of Proposition~\ref{prop:rel-G-cr}, the definition of a completely reducible action makes no reference to the ambient group $G$, essentially requiring only the fact that $G$ is an algebraic variety containing the group $HK$ as a closed subvariety. Thus the independence of relative complete reducibility from the ambient algebraic group indicated in \cite[Corollary 3.6]{Bate2011} becomes an intrinsic feature of the definition in this case.
\end{remark}

This feature is demonstrated even more prominently in the following 
geometric characterization of relative complete reducibility, where, thanks to Proposition~\ref{prop:rel-G-cr}, there is no longer any reference to an ambient reductive group $G$.

Let $H$ be a subgroup of $G$ and let $G\hookrightarrow\GL_m$
be an embedding of algebraic groups.
Then $\mathbf{h} \in H^n$ is
called a \emph{generic tuple of $H$ for the embedding
	$G\hookrightarrow\GL_m$} if $\mathbf{h}$ generates the
associative subalgebra of $\Mat_m$ spanned by $H$.
We call $\mathbf{h}\in H^n$ a \emph{generic tuple of $H$}
if it is a generic tuple of $H$ for some embedding $G\hookrightarrow\GL_m$,
\cite[Definition 5.4]{Bate2013}.

Relative complete reducibility has a natural characterisation in terms of $K$-orbits in $G^{n}$, where $K$ acts diagonally on $G^{n}$ by simultaneous conjugation: Let $\mathbf{h} \in G^{n}$ be a generic tuple for the subgroup $H$. Then $H$ is relatively $G$-cr with respect to $K$ if and only if the $K$-orbit $K \cdot \mathbf{h}$ is closed in $G^{n}$ \cite[Theorem 3.5(iii)]{Bate2011}. Therefore in the case that $H$ is also an algebraic group acting morphically on $K$, so that the semidirect product $H \ltimes K$ is again an algebraic group, embedding this into some large (arbitrary) reductive group $G$ gives the following.

\begin{corollary}
	Suppose that $H$ is a linear algebraic group acting morphically on a (not necessarily connected) reductive algebraic group $K$. Let $\mathbf{h} \in H^{n}$ be a generic tuple for $H$. Then $H$ acts completely reducibly on $K$ if and only if the $K$-orbit $K \cdot \mathbf{h}$ is Zariski closed as a subset of $(H \ltimes K)^{n}$.
\end{corollary}

\subsection{Complete reducibility and Steinberg endomorphisms.} \label{sec:Sigma}
Let $G$ be connected and let 
$\sigma: G \rightarrow G$ be a 
\emph{Steinberg endomorphism} of $G$, i.e.\ 
a surjective endomorphism of $G$ 
that fixes only finitely many points, see \cite{Steinberg68} for a detailed discussion.
Steinberg endomorphisms of $G$ belong to the set of all 
isogenies $G\rightarrow G$ (see \cite[7.1(a)]{Steinberg68})
which encompasses in particular all (generalized) Frobenius endomorphisms, 
i.e.\ endomorphisms of $G$ some power of which are Frobenius endomorphisms 
corresponding to some $\mathbb{F}_q$-rational structure on $G$. 
We recall the notion of $\sigma$-complete reducibility from 
\cite{Herpel2011}.

\begin{definition}
	\label{def:sigma-cr}
	Let $\sigma$ be a Steinberg endomorphism of $G$ and let $H$ be a subgroup of $G$. We say that $H$ is 
	\emph{$\sigma$-completely reducible} (\emph{$\sigma$-cr} for short), provided that whenever 
	$H$ lies in a $\sigma$-stable parabolic subgroup 
	$P$ of $G$, it lies in a $\sigma$-stable Levi subgroup of $P$.
\end{definition}

The notions of $G$-complete reducibility, etc., can be extended to reductive groups defined
over arbitrary fields; see \cite{Serre2003-2004}, \cite[\S~5]{Bate2005}.
The concept in Definition \ref{def:sigma-cr} is motivated as follows: 
If $\sigma_q$ is a standard Frobenius morphism of $G$, 
then a subgroup $H$ of $G$ is defined over $\mathbb{F}_q$ if and only if it is $\sigma_q$-stable and if so, 
$H$ is $G$-completely reducible 
over $\mathbb{F}_q$ if and only if it is $\sigma_q$-completely reducible.
The following is the main theorem from \cite{Herpel2011}; it is a generalization of 
a special case of the rationality result 
\cite[Theorem~5.8]{Bate2005} to arbitrary
Steinberg endomorphisms of $G$.

\begin{theorem}[{\cite[Theorem 1.4]{Herpel2011}}]
	\label{thm:sigma-cr}
	Let $\sigma$ be a Steinberg endomorphism of $G$ 
	and let $H$ be a $\sigma$-stable subgroup of $G$. Then $H$ is $\sigma$-completely reducible if and only if $H$ is $G$-completely reducible.
\end{theorem}

We now come to part \ref{thm:gen-sigma} of Theorem~\ref{thm:generalisation}. 

\begin{proposition}
	\label{prop:Sigma}
		Let $\sigma$ be a Steinberg endomorphism of the connected reductive group $G$ and let $H$ be a subgroup of $G$. Let $\Gamma$ be the subgroup of $\Aut(X(G))$ generated by $\sigma$ and the image of $H$. Then $H$ is $\sigma$-completely reducible if and only if $\Gamma$ acts completely reducibly on $G$.
\end{proposition}

\begin{proof}
	By definition of $\Gamma$, the fixed-point subcomplex $X(G)^{\Gamma}$ consists of those $\sigma$-stable parabolic subgroups of $G$ containing $H$. Thus the statement that every member of $X(G)^\Gamma$ has an opposite in $X(G)^\Gamma$ is trivially equivalent to the statement that whenever $H$ lies in a $\sigma$-stable parabolic subgroup $P$ of $G$, it lies in a $\sigma$-stable parabolic subgroup $P^{\rm op}$ opposite to $P$; when this holds, $H$ is contained in the intersection of $P$ and $P^{\rm op}$, which is a $\sigma$-stable Levi subgroup of $P$.
	
	Conversely, suppose that whenever $H$ is contained in a $\sigma$-stable parabolic subgroup $P$ of $G$, it is contained in a $\sigma$-stable Levi subgroup $L$ of $P$. Being contained in $L$ implies that $H$ lies in the unique parabolic subgroup $Q$ of $G$ whose intersection with $P$ is $L$. But since $P$ and $L$ are $\sigma$-stable, the image of $Q$ under $\sigma$ is again a parabolic subgroup of $G$ whose intersection with $P$ is $L$; uniqueness now implies that $Q$ is $\sigma$-stable. Thus every member of $X(G)^\Gamma$ has an opposite in $X(G)^\Gamma$.	
\end{proof}

\subsection{Complete reducibility for finite groups of Lie type.} \label{sec:GsigmaCR}

Suppose $G$ is connected reductive and defined over a finite field, hence equipped with a 
Steinberg endomorphism $\sigma$.  
The fixed point subgroup $G_\sigma :=\{g \in G \mid \sigma(g) =g\}$ of $\sigma$
is thus a \emph{finite group of Lie type}.
Likewise, for a $\sigma$-stable subgroup $M$  of $G$, let $M_\sigma := M \cap G_\sigma$ be its fixed point subgroup.
As a further variant of the concepts above, 
it is natural to consider the following notion 
of complete reducibility for the finite reductive groups $G_\sigma$.

\begin{definition}
	\label{def:Gsigma-cr}
	Let $\sigma$ be a Steinberg endomorphism of $G$ and let $H$ be a subgroup of $G_\sigma$.
	Then $H$ is \emph{$G_\sigma$-completely reducible} (\emph{$G_\sigma$-cr} for short) provided if $H \le P_\sigma$ for some  $\sigma$-stable parabolic subgroup $P$ of $G$, then $H \le L_\sigma$ for some  $\sigma$-stable Levi subgroup $L$ of $P$.
\end{definition}

An easy application of Theorem \ref{thm:sigma-cr} gives that 
the notion of 
$G_\sigma$-complete reducibility is already captured by the usual concept in the ambient reductive group $G$.

\begin{theorem}
		\label{thm:Gsigma-cr}
		Let $\sigma$ be a Steinberg endomorphism of $G$ and let $H$ be a subgroup of $G_\sigma$.
		Then $H$ is $G_\sigma$-completely reducible if and only if $H$ is $G$-completely reducible.
\end{theorem}

\begin{proof}
	If $H$ is $G$-cr and $H \le P_\sigma$ for some  $\sigma$-stable parabolic subgroup $P$ of $G$, then by Theorem \ref{thm:sigma-cr}, there is a $\sigma$-stable Levi subgroup $L$ of $P$ containing $H$. Thus $H \le L \cap G_\sigma = L_\sigma$.
	
	If $H$ is not $G$-cr, then by Theorem \ref{thm:sigma-cr}, it is not $\sigma$-cr.
	Thus there is a (proper) $\sigma$-stable parabolic subgroup $P$ of $G$ containing $H$ but no $\sigma$-stable Levi subgroup of $P$ contains $H$. In particular, $H \le P_\sigma$, but $H$ does not lie in $L_\sigma$ for any $\sigma$-stable Levi subgroup $L$ of $P$.	
\end{proof}

Theorem~\ref{thm:Gsigma-cr} allows us to derive statements for $G_\sigma$-complete reducibility from 
corresponding ones for the ambient reductive group $G$. Note that the concept of $G_\sigma$-complete reducibility relies not just on the group structure of $G_\sigma$, but also on information on the embedding $G_\sigma \to G$. For instance, the
isomorphism $\PSL_2(4) \cong \PSL_2(5)$, respectively $\PSL_2(7) \cong \PSL_3(2)$,
 leads to two different collections of `completely reducible' subgroups depending on whether we consider it as a group in characteristic $2$ or $5$, respectively $7$ or $2$.

In our next example, instances of finite subgroups of $G$ readily lead to (non) $G_\sigma$-completely reducible subgroups of $G_\sigma$.
\begin{example}
	\label{ex:g2}
	Suppose $k$ is of characteristic $2$. Let $G$ be a simple algebraic group of type $G_2$ over $k$. Let $M$ be a maximal rank subgroup of type $\tilde A_1 A_1$. In 
	\cite[\S 7]{Bate2010} a family of finite subgroups $H_a$ of $M$ was constructed for $a \in k^*$ with the property that $H_a$ is $G$-cr but not $M$-cr, \cite[Proposition 7.17]{Bate2010}. Since $H_a$ is isomorphic to the finite symmetric group $S_3$ and a fixed $a \in k^*$ belongs to $\mathbb {F}_q$ for some suitable power $q$ of $2$, we see that $H_a$ belongs to 
	$M_{\sigma}$, where $\sigma$ denotes the standard Frobenius endomorphism of $G$ associated with the $\mathbb {F}_q$-structure of $G$. It follows from Theorem~\ref{thm:Gsigma-cr} that $H_a$ is $G_{\sigma}$-cr but not $M_{\sigma}$-cr for $a \in k^*$.   
\end{example}

Our next result gives a general Clifford Theorem for finite groups of Lie type.
Its proof is immediate from Theorem~\ref{thm:Gsigma-cr} and \cite[Theorem 3.10]{Bate2005}.
 
 \begin{corollary}
\label{cor:Gsigma-cr-normal}
Let $\sigma$ be a Steinberg endomorphism of $G$ and let $H$ be a subgroup of $G_\sigma$ and $N$ a normal subgroup of $H$.
If $H$ is $G_\sigma$-completely reducible, then so is $N$.
 \end{corollary}
 
 In particular, it follows from Corollary \ref{cor:Gsigma-cr-normal} that if $H$ is a subgroup of $G_\sigma$ so that $N_{G_\sigma}(H)$ is $G_\sigma$-cr, then so is $H$. The converse is less clear.

Let $\sigma$ and $\tau$ be Steinberg endomorphisms of $G$ so that $G_\sigma \subseteq G_\tau$.
Let $H$ be a subgroup of $G_\sigma$. Then by Theorem \ref{thm:Gsigma-cr}, $H$ is $G_\sigma$-cr if and only if it is $G_\tau$-cr. 

\begin{theorem}
	\label{thm:Gsigma-cr-normaliser}
	Let $\sigma$ be a Steinberg endomorphism of $G$ and let $H$ be a subgroup of $G_\sigma$. Then there exists a Steinberg endomorphism $\tau$  of $G$ with $G_\sigma \subseteq G_\tau$ such that $H$ is $G_\sigma$-completely reducible if and only if $N_{G_\tau}(H)$ is $G_\tau$-completely reducible.
\end{theorem}

\begin{proof}
	Suppose that $H$ is $G_\sigma$-cr. Then by Theorem \ref{thm:Gsigma-cr} and \cite[Corollary 3.16]{Bate2005}, $N_G(H)$ is $G$-cr. Thanks to \cite[Lemmma 2.10]{Bate2005}, there is a finitely generated subgroup $\Gamma$ of $N_G(H)$ with the property that $\Gamma$ lies in the very same parabolic and Levi subgroups of $G$ as $N_G(H)$. 
	In particular, since $N_G(H)$ is $G$-cr, so is $\Gamma$. As a finitely generated subgroup of $G$, $\Gamma$ is finite. Thus $\Gamma \le G_\tau$ for some Steinberg endomorphism $\tau$ of $G$. Without loss we may also assume that $G_\sigma \subseteq G_\tau$. 
	
	Since $\Gamma \le G_\tau$ is $G$-cr, it is $G_\tau$-cr, by Theorem \ref{thm:Gsigma-cr}.
	
	Now let $N_{G_\tau}(H) = N_G(H) \cap G_\tau$ belong to $P_\tau$ for some $\tau$-stable parabolic subgroup $P$ of $G$. Then $\Gamma \le N_G(H) \cap G_\tau  = N_{G_\tau}(H)  \le P_\tau \le P$.
	Since $\Gamma$ is $G_\tau$-cr, there is a $\tau$-stable Levi $L$ of $P$ so that $\Gamma \le L_\tau \le L$. The properties of $\Gamma$ imply that $N_{G_\tau}(H) = N_G(H) \cap G_\tau  \le  L \cap G_\tau = L_\tau$. Consequently, $N_{G_\tau}(H)$ is $G_\tau$-cr.
	
	Conversely, if $N_{G_\tau}(H)$ is $G_\tau$-cr, then $H$ is $G_\tau$-cr, by Corollary \ref{cor:Gsigma-cr-normal}. Finally, $H$ is $G_\sigma$-cr by the comment above.
\end{proof}

The following is an immediate consequence of Corollary \ref{cor:Gsigma-cr-normal} and Theorem \ref{thm:Gsigma-cr-normaliser}.

\begin{corollary}
	\label{cor:Gsigma-cr-centraliser}
	Let $\sigma$ be a Steinberg endomorphism of $G$ and let $H$ be a subgroup of $G_\sigma$. Then there exists a Steinberg endomorphism $\tau$  of $G$ with $G_\sigma \subseteq G_\tau$ such that if $H$ is $G_\sigma$-completely reducible, then $C_{G_\tau}(H)$ is $G_\tau$-completely reducible.
\end{corollary}

\subsection{Complete reducibility and building automorphisms.} \label{sec:A} 

We present some complements to our main development.
Let $G$ be connected reductive and let 
$\Sigma$ be a group of building automorphisms of the building $X(G)$ of $G$ (not necessarily type-preserving). 
If one restricts attention to $\Sigma$-stable parabolic subgroups and their $\Sigma$-stable opposites, 
one obtains the following more general notion. 

\begin{definition} \label{def:acr}
	Let $H$ be a subgroup of the connected reductive algebraic group $G$, and let $\Sigma$ be a group of building automorphisms of $X(G)$. We say that $H$ is \emph{$\Sigma$-completely reducible} (\emph{$\Sigma$-cr} for short) if, whenever $H$ is contained in an $\Sigma$-stable parabolic subgroup $P$ of $G$, it is contained in a $\Sigma$-stable parabolic opposite to $P$.
	\end{definition}

Note that Definition \ref{def:sigma-cr} is just 
the special case $\Sigma = \left<\sigma\right>$ for a Steinberg endomorphism $\sigma$ of $G$ in Definition \ref{def:acr}.

Then incursions into this concept may be made 
analogous to the ones from Section  \ref{sec:Sigma}.
We indicate the counterpart of Proposition \ref{prop:Sigma}, which shows that Definition \ref{def:acr} is equivalent to  
$X(G^{\circ})^\Gamma$ being $X$-cr, where $\Gamma$ is the subgroup of $\Aut(X(G^{\circ}))$ generated by $\Sigma$ and the image of the subgroup $H$ of $G$ in question.
The proof of the latter applies \emph{mutatis mutandis}
arguing via $\Sigma$-stable opposite parabolic subgroups in place of $\sigma$-stable Levi subgroups and 
is left to the reader.

\begin{proposition} \label{prop:A}
	Let $G$ be connected reductive and let 
	$H$ be a subgroup of $G$, let $\Sigma \le \Aut(X(G))$ and let $\Gamma$ be the subgroup of $\Aut(X(G))$ generated by $\Sigma$ and the image of $H$. Then $H$ is $\Sigma$-completely reducible if and only if $\Gamma$ acts completely reducibly on $X(G)$.
\end{proposition}

\begin{proof}
	Suppose $\Gamma$ acts completely reducibly on $X(G)$, and suppose $H \le P$, a $\Sigma$-stable parabolic subgroup of $G$. Then $P$ is $\Gamma$-stable. So by hypothesis, $\Gamma$ fixes an opposite $Q$ of $P$. In particular, $Q$ is $\Sigma$-stable. 
	
	Conversely, suppose that $H$ is $\Sigma$-cr. 
	Let $P \in X(G)^{\Gamma}$. So $P$ is $\Sigma$-stable. Then by hypothesis there is a $\Sigma$-stable opposite parabolic $Q$ in $X(G)$ containing $H$. Thus $Q$ is $\Gamma$-stable. 
\end{proof}

\subsection{Complete reducibility for arbitrary \texorpdfstring{$k$}{k}.} \label{sec:rat} 

If $G$ is defined over an arbitrary field $k$, then the collection of $k$-defined parabolic subgroups of $G^{\circ}$ also gives a spherical building (cf.~\cite[\S 5]{Tits74}), which we denote by $X_k(G^{\circ})$. 
If $k$ is algebraically closed, then we set
$X(G^{\circ}) = X_{k}(G^{\circ})$.

In \cite[\S 5]{Tits74}, Tits showed that for $G = G^{\circ}$, if $X_{k}(G)$ is irreducible and of rank at least 2 then every automorphism of $X_{k}(G)$ arises via some natural constructions involving the algebraic group $G$: the building automorphisms induced by isogenies of $G$ and by field automorphisms of $k$.

The last part of 
Definition \ref{def:cr-action} extends naturally to this rational setting.

\begin{definition}[{\cite[\S 2]{Serre2003-2004}}] \label{def:cr-action-rat} 
	Suppose $G$ is defined over an arbitrary field $k$. Let $X = X_{k}(G^{\circ})$ be the spherical building of $G^{\circ}$. Then an action of $\Gamma$ on $G$ by $k$-automorphisms is called \emph{completely reducible} if the induced action on $X$ is completely reducible
	(in the sense of Definition \ref{def:cr-action}).
\end{definition}

Let $X_{k}(G)$ be the set of R-parabolic $k$-subgroups of $G$. Note that by assumption $\Gamma$ acts on $X_{k}(G)$.
For the rational counterpart of Definition~\ref{def:relativeCR}, we replace the R-parabolic and R-Levi subgroups by $k$-defined R-parabolic and $k$-defined R-Levi subgroups, see \cite[Definition 4.1]{Bate2011}.
We obtain the following rational counterparts of Theorem \ref{thm:generalisation}\ref{thm:nonconnected} and \ref{thm:gen-rel}.

\begin{theorem} \label{thm:generalisation-rat} %
	Let $G$ be a (possibly non-connected) reductive algebraic group defined over~$k$.
	\begin{enumerate}[label=(\roman*)]
		\item The action of an abstract group $\Gamma$ by means of $k$-automorphisms on $G$ is completely reducible in the sense of Definition~\ref{def:cr-action-rat} if and only if every member of $X_{k}(G)^{\Gamma}$ has an opposite in $X_{k}(G)^\Gamma$. \label{thm:nonconnected-rat}
		\item Let $K$ be a $k$-defined reductive subgroup of $G$ and suppose that $N_G(K^\circ)$ and $C_G(K^\circ)$ are $k$-defined.
		If $H \le N_G(K^{\circ})$, then $H$ is relatively $G$-completely reducible over $k$ with respect to $K$  if and only if the induced action of $H$ on $X_k(K^{\circ})$ is completely reducible. \label{thm:gen-rel-rat}
	\end{enumerate}
\end{theorem}

For the proof of \ref{thm:nonconnected-rat}, follow the proof of Proposition \ref{prop:G-cr-G0-cr}, replacing the R-parabolic subgroups by $k$-defined R-parabolic subgroups. 
Note that $N_G(P)$ is also $k$-defined, by the proof of \cite[Proposition 5.4(a)]{Martin03}.
For \ref{thm:gen-rel-rat} replace Theorem \ref{thm:generalisation}\ref{thm:nonconnected} by part \ref{thm:nonconnected-rat} and \cite[Theorem 1]{Gruchot2019} by its rational version \cite[Theorem 7.5]{Gruchot2019}. The rational counterpart of Lemma \ref{lem:equivariant} does not hold in general, see \cite[Example 3.10]{UCHIYAMA2016}. However, the map $K\to K/Z(K)$ is central. Without loss, we can assume that $K$ is connected. Then the image resp.\ preimage of a $k$-defined parabolic subgroup is $k$-defined, by \cite[22.6 Theorem(i)]{Borel1991}.
Clearly, the image of a $k$-defined Levi subgroup is $k$-defined. 
The preimage of a $k$-defined Levi subgroup is $k$-defined,  by \cite[22.5 Corollary]{Borel1991}. This suffices to adapt the proof of Theorem \ref{thm:generalisation}\ref{thm:gen-rel} to the rational setting.

\subsection{Complete reducibility and the topology of fixed point subcomplexes of \texorpdfstring{$X(G^{\circ})$}{X(G0)}}\label{sec:top}
In our final section we consider 
the connection between the notion of 
$G$-complete reducibility and the 
geometric realization of the 
building of $G^\circ$ due to Serre \cite[Theorem~2]{Serre1997}. 

As before, let $X = X(G^{\circ})$.
Take a Levi subgroup $L$ of $G$ and set $s(L) := X^L$ denote the 
subcomplex of $X$ consisting of the parabolic subgroups of $G^{\circ}$
containing $L$.
For every parabolic subgroup $P$ in $s(L)$ there is a unique 
Levi subgroup $M$ of $P$ with $L\subseteq M$.
Moreover, the parabolic subgroup $P^-$  
such that $P\cap P^- = M$ is also contained in $s(L)$, so that
$P$ has an opposite in $s(L)$;
thus $s(L)$ is $X$-cr.
This argument also shows that each $P$ in $s(L)$ has a unique opposite
in $s(L)$, and this implies that the
geometric realization of $s(L)$ has the homotopy type of a 
single sphere (cf.\ Theorem \ref{thm:gcr-building}(v) below). 
Serre calls the subcomplexes $s(L)$ of $X$ \emph{Levi spheres}, 
\cite[\S 2]{Serre1997} or 
\cite[2.1.6, 3.1.7]{Serre2003-2004}.
The following is part of \cite[Theorem 2]{Serre1997} 
in our context and applies to each of the notions of 
complete reducibility discussed above, thanks to
Theorem~\ref{thm:generalisation}.

\begin{theorem}[{\cite[Theorem~2]{Serre1997}}]
	\label{thm:gcr-building}
	Let $G$ be reductive and defined over the field $k$, $X = X_{k}(G^{\circ})$ and let $\Gamma \le \Aut(X)$. 
	Then the following	are equivalent:
	\begin{itemize}
		\item[(i)] $\Gamma$ acts completely reducibly on $X$;
		\item[(ii)] $X^\Gamma$ is $X$-completely reducible;
		\item[(iii)] $X^\Gamma$ contains a Levi sphere of the same dimension as $X^\Gamma$; 
		\item[(iv)] $X^\Gamma$ is not contractible 
		(i.e.\ does not have the homotopy type of a point);
		\item[(v)] 
		$X^\Gamma$ has the homotopy type of a bouquet of spheres.
	\end{itemize}
\end{theorem}

\bigskip
Acknowledgements: We would like to thank Michael Bate for helpful discussions on the material of this note.
 
\providecommand{\bysame}{\leavevmode\hbox to3em{\hrulefill}\thinspace}
\providecommand{\MR}{\relax\ifhmode\unskip\space\fi MR }
\providecommand{\MRhref}[2]{%
	\href{http://www.ams.org/mathscinet-getitem?mr=#1}{#2}
}
\providecommand{\href}[2]{#2}

\end{document}